\documentclass{amsart}%
\usepackage{amsfonts}
\usepackage{amsmath}
\usepackage{amssymb}
\usepackage{graphicx}%
\providecommand{\U}[1]{\protect\rule{.1in}{.1in}}
\newtheorem{theorem}{Theorem}
\theoremstyle{plain}

\newtheorem{corollary}{Corollary}

\newtheorem{lemma}{Lemma}

\newtheorem{remark}{Remark}

\numberwithin{equation}{section}
\begin{document}
\title[ ]{Korovkin type theorems for weakly nonlinear and monotone operators}
\author{Sorin G. Gal}
\address{Department of Mathematics and Computer Science\\
University of Oradea\\
University\ Street No. 1, Oradea, 410087, Romania\\
and Academy of Romanian Scientists, Splaiul Independentei nr. 54, 050094,
Bucharest, Romania}
\email{galso@uoradea.ro, galsorin23@gmail.com}
\author{Constantin P. Niculescu}
\address{Department of Mathematics, University of Craiova\\
Craiova 200585, Romania and The Institute of Mathematics of the Romanian
Academy, Bucharest, Romania}
\email{constantin.p.niculescu@gmail.com}
\date{May 3, 2022}
\subjclass[2000]{41A35, 41A36, 41A63}
\keywords{Korovkin type theorems, monotone operator, sublinear operator, convergence
almost everywhere, convergence in measure, convergence in $L^{p}$-norm,
Choquet's integral,.}

\begin{abstract}
In this paper we prove analogues of Korovkin's theorem in the context of
weakly nonlinear and monotone operators acting on Banach lattices of functions
of several variables. Our results concern the convergence almost everywhere,
the convergence in measure and the convergence in $L^{p}$-norm. Several
results illustrating the theory are also included.

\end{abstract}
\maketitle

\section{Introduction}

Korovkin's theorem \cite{Ko1953}, \cite{Ko1960} provides a very simple test of
convergence to the identity for any sequence $(T_{n})_{n}$ of positive linear
operators that map $C\left(  [0,1]\right)  $ into itself: the occurrence of
this convergence for the functions $1,~x$ and $x^{2}$. In other words, the
fact that%
\[
\lim_{n\rightarrow\infty}T_{n}(f)=f\text{\quad uniformly on }[0,1]
\]
for every $f\in C\left(  [0,1]\right)  $ reduces to the status of the three
aforementioned functions. Due to its simplicity and usefulness, this result
has attracted a great deal of attention leading to numerous generalizations.
Part of them are included in the authoritative monograph of Altomare and
Campiti \cite{AC1994}. and the excellent survey of Altomare \cite{Alt2010}.
For some very recent contributions to this topics see \cite{Alt2021},
\cite{Alt2021b} and \cite{Popa}.

Recently, the present authors have extended the Korovkin theorem to the
framework of sublinear and monotone operators acting on function spaces
endowed with the topology of uniform convergence on compact sets. See
\cite{Gal-Nic-Med}, \cite{Gal-Nic-Aeq} and \cite{Gal-Nic-RACSAM}.

The aim of the present paper is to prove that similar results hold in the
context of the three usual modes of convergence used in measure theory:
convergence almost everywhere, convergence in probability and the convergence
in $L^{p}$ norm (also known as the convergence in $p$-mean).

The necessary background on our nonlinear framework is summarized in Section
1. We deal with the class of sublinear and monotone operators acting on Banach
lattices of functions which verify the property of translatability relative to
the multiples of unity. This was introduced in \cite{Gal-Nic-RACSAM},
motivated by our interest in Choquet's theory of integration, but there many
other examples outside that theory, mentioned in this paper. The continuity of
sublinear operators can be characterized via the notion of norm as in the
linear case. As far as we know, the result of Theorem \ref{thmKrein} asserting
that every sublinear and monotone is Lipschitz continuous is new. It extends a
classical result due M.G. Krein \cite{Krein} who considered the case of linear functionals.

The extension of the Korovkin theorem for the aforementioned modes of
convergence makes the object of Theorem 1 (Section 3) and respectively of
Theorem 2 and Theorem 3 (Section 4). As a consequence of Theorem 1 we provide
in Section 5 an alternative proof of the Lebesgue differentiation theorem. The
pointwise convergence of the Bernstein-Kantorovich sequences and the
Sz\'{a}sz-Mirakjan-Kantorovich sequences (as well as of their Choquet
counterparts obtained by replacing the Lebesgue integral by Choquet integral
with respect to a submodular capacity) is presented in Section 6. Our results
in this respect extend a classical theorem of Lorentz \cite{Lor} (Theorem
2.1.1, p. 30).

The paper ends with a section of further results and comments.

\section{Preliminaries on nonlinear operators}

In what follows we denote by $X$ a metric measure space, that is, a triple
$(X,d,\mu)$ consisting of a space $X$ endowed with the metric $d$ and the
measure $\mu,$ defined on the sigma algebra $\mathcal{B}(X)$ of Borel subsets
of $X$. Notice that every set can be turned into a metric measure space by
considering on it the discrete metric and any finite combination (with
positive coefficients) of Dirac measures.

Attached to $X$ is the vector lattice $\mathcal{F}(X)$ of all real-valued
functions defined on $X$, endowed with the pointwise ordering. Among the
vector sublattices of $\mathcal{F}(X)$ which play a role in the extension of
Korovkin's results we mention here
\begin{align*}
\mathcal{F}_{b}(X)  &  =\left\{  f\in\mathcal{F}(X):\text{ }f\text{
bounded}\right\} \\
C(X)  &  =\left\{  f\in\mathcal{F}(X):\text{ }f\text{ continuous}\right\}  ,\\
C_{b}(X)  &  =\left\{  f\in\mathcal{F}(X):\text{ }f\text{ continuous and
bounded}\right\}  ,\\
\mathcal{U}C_{b}(X)  &  =\left\{  f\in\mathcal{F}(X):\text{ }f\text{ uniformly
continuous and bounded}\right\}  ,
\end{align*}
and
\[
C\,_{c}(X)=\left\{  f\in\mathcal{F}(X):\text{ }f\text{ continuous and having a
compact support}\right\}  ,
\]
to which one should add the usual sublattices of measurable functions,%
\begin{align*}
\mathcal{M}(X)  &  =\left\{  f\in\mathcal{F}(X):\text{ }f\text{ Borel
measurable}\right\}  \text{,}\\
\mathcal{AC}_{b}(X)  &  =\left\{  f\in\mathcal{F}(X):\text{ }f\text{ bounded
and almost everywhere continuous}\right\}
\end{align*}
as well as all Banach function spaces (including the spaces $L^{p}(\mu)$ with
$p\in\lbrack1,\infty]$). A thorough presentation of Banach function spaces can
be found in the book by Bennett and Sharpley \cite{BS}.

Notice that $C(X),\mathcal{AC}_{b}(X)$ and $L^{p}(\mu)$ are sublattices of
$\mathcal{M}(X).$ The spaces $C(X),C_{b}(X)~$and $\mathcal{U}C_{b}(X)$
coincide when $X$ is a compact metric space. $\mathcal{A}C_{b}(X)$ coincides
with the space of Riemann integrable functions when $X$ is a compact
$N$-dimensional interval of $\mathbb{R}^{N}$ and $\mu$ is the Lebesgue
measure. Notice also that the spaces $C_{b}(X)$, $\mathcal{U}C_{b}(X)$ and
$L^{p}(\mu)$ are Banach lattices with respect to appropriate norms, precisely,
the sup norm
\[
\left\Vert f\right\Vert _{\infty}=\sup\left\{  \left\vert f(x)\right\vert
:x\in X\right\}  ,
\]
in the case of the first two spaces and the $L^{p}$-norm,%
\[
\left\Vert f\right\Vert _{p}=\left\{
\begin{array}
[c]{cl}%
\left(  \int_{X}\left\vert f(x)\right\vert ^{p}\mathrm{d}\mu(x)\right)  ^{1/p}
& \text{if }p\in\lbrack1,\infty)\\
\operatorname*{esssup}\limits_{x\in X}\left\vert f(x)\right\vert  & \text{if
}p=\infty
\end{array}
\right.
\]
in the case of spaces $L^{p}(\mu).$

An important family of Lipschitz continuous functions in $C(X)$ is that
associated to the metric $d$ by the formulas
\[
d_{x}:X\rightarrow\mathbb{R},\text{\quad}d_{x}(y)=d(x,y)\quad(x\in X).
\]

See \cite{CN2014} and \cite{MN} for the necessary background on Banach
lattices used in this paper.

As is well known, all norms on the $N$-dimensional real vector space
$\mathbb{R}^{N}$ are equivalent. When endowed with the sup norm and the
coordinate wise ordering, $\mathbb{R}^{N}$ can be identified (algebraically,
isometrically and in order) with the Banach lattice $C\left(  \left\{
1,...,N\right\}  \right)  $, where $\left\{  1,...,N\right\}  $ carries the
discrete topology.

Suppose that $X$ and $Y$ are two metric spaces and $E$ and $F$ are
respectively ordered vector subspaces (or subcones of the positive cones) of
$\mathcal{F}(X)$ and $\mathcal{F}(Y)$ that contain the unity. An operator
$T:E\rightarrow F$ is said to be a \emph{weakly nonlinear} if it satisfies the
following two conditions:

\begin{enumerate}
\item[(SL)] (\emph{Sublinearity}) $T$ is subadditive and positively
homogeneous, that is,%
\[
T(f+g)\leq T(f)+T(g)\quad\text{and}\quad T(\alpha f)=\alpha T(f)
\]
for all $f,g$ in $E$ and $\alpha\geq0;$

\item[(TR)] (\emph{Translatability}) $T(f+\alpha\cdot1)=T(f)+\alpha T(1)$ for
all functions $f\in E$ and all numbers $\alpha\geq0.$
\end{enumerate}

In the case when $T$ is \emph{unital} (that is, $T(1)=1)$ the condition of
translatability takes the form%
\begin{equation}
T(f+\alpha\cdot1)=T(f)+\alpha1, \label{TR+!}%
\end{equation}
for all $f\in E$ and $\alpha\geq0.$

In this paper we are especially interested in those weakly nonlinear operators
which preserve the ordering, that is, which verify the following condition:

\begin{enumerate}
\item[(M)] (\emph{Monotonicity}) $f\leq g$ in $E$ implies $T(f)\leq T(g)$ for
all $f,g$ in $E.$
\end{enumerate}

Some authors prefer the term of \emph{isotonicity} for monotonicity in order
to avoid any confusion with the monotonicity of subdifferentials in convex
analysis. However in our paper we don't touch the problem of differentiability.

Crandall and Tartar \cite{CT} have noticed that if an operator $T:L^{\infty
}(\mu)\rightarrow L^{\infty}(\mu)$ is translatable and unital, then it is
monotone if, and only if, it is Lipschitz with Lipschitz constant at most 1.
In particular, this remark holds for $\mathbb{R}^{N}$ endowed with the
sup-norm and any operator $T:\mathbb{R}^{N}\rightarrow\mathbb{R}^{N}$ which
verifies the property (\ref{TR+!}).

Real analysis and harmonic analysis offer numerous example of weakly nonlinear
and monotone operators. So is the operator $T$ acting from the Lebesgue space
$L^{p}(\mathbb{R}^{N})$\ $(p>1)$ into itself via the formula%
\begin{equation}
\left(  Tf\right)  (x)=\sup_{r>0}\frac{1}{\operatorname*{vol}(B_{r}(x))}%
\int_{B_{r}(x)}f(y)\mathrm{d}y.
\end{equation}

The same holds for the operator $S:L^{1}(\mu)\rightarrow L^{1}(\mu),$
associated to a Borel probability measure $\mu$ on $(0,1),$ and defined by the
formula%
\[
(Sf)(t)=\sup_{\mu(A)\leq t}\frac{1}{\mu(A)}\int_{A}f(s)\mathrm{d}\mu.
\]
Notice that $S$ is also unital.

Many more examples can be found in our paper \cite{Gal-Nic-RACSAM}.

Meantime, it is important to notice the existence of sublinear operators which
are neither monotone nor translatable, and example being the Hardy--Littlewood
maximal operator $M:L^{p}(\mathbb{R}^{N})\rightarrow L^{p}(\mathbb{R}^{N}%
)$\ $(p>1),$ defined by the formula%
\[
Mf(x)=\sup_{r>0}\frac{1}{\operatorname*{vol}(B_{r}(x))}\int_{B_{r}%
(x)}\left\vert f(y)\right\vert \mathrm{d}y.
\]
However, $M$ is monotone and translatable on the positive cone of
$L^{p}(\mathbb{R}^{N}).$

A stronger condition than translatability is that of \emph{comonotonic
additivity},

\begin{enumerate}
\item[(CA)] $T(f+g)=T(f)+T(g)$ whenever the functions $f,g\in E$ are
comonotone in the sense that%
\[
(f(s)-f(t))\cdot(g(s)-g(t))\geq0\text{\quad for all }s,t\in X.
\]

\end{enumerate}

The condition of comonotonic additivity occurs naturally in the context of
Choquet's integral (and thus in the case of Choquet type operators, which are
sublinear, comonotonic additive and monotone). See \cite{Gal-Nic-Aeq} and
\cite{Gal-Nic-JMAA} as well as the references therein.

Suppose that $E$ and $F$ are two Banach lattices and $T$ $:E\rightarrow F$ is
an operator (not necessarily linear or continuous).

If $T$ is positively homogeneous operator, then
\[
T(0)=0.
\]
As a consequence, every positively homogeneous and monotone operator $T$ maps
positive elements into positive elements,%
\begin{equation}
Tx\geq0\text{\quad for all }x\geq0. \label{pos-op}%
\end{equation}
Consequently, for linear operators the property (\ref{pos-op}) is equivalent
to monotonicity.

Every sublinear operator is convex and a convex operator is sublinear if and
only if it is positively homogeneous.

The notion of norm can be introduced for every continuous sublinear operator
$T:E\rightarrow F$ via the formulas%
\begin{align*}
\left\Vert T\right\Vert  &  =\inf\left\{  \lambda>0:\left\Vert T\left(
f\right)  \right\Vert \leq\lambda\left\Vert f\right\Vert \text{ for all }f\in
E\right\} \\
&  =\sup\left\{  \left\Vert T(f)\right\Vert :f\in E,\text{ }\left\Vert
f\right\Vert \leq1\right\}  .
\end{align*}
A sublinear operator may be discontinuous, but when it is continuous, it is
Lipschitz continuous:

\begin{lemma}
\label{lemcont}If $T:E\rightarrow F$ is a continuous sublinear operator, then
\[
\left\Vert T\left(  f\right)  -T(g)\right\Vert \leq2\left\Vert T\right\Vert
\left\Vert f-g\right\Vert \text{\quad for all }f\in E.
\]

\end{lemma}

\begin{proof}
Indeed, for all $f,g\in E$ we have $T(f)\leq T(g)+T(f-g)\leq T(g)+\left\vert
T(f-g)\right\vert $ and $T(g)\leq T(f)+\left\vert T(g-f)\right\vert ,$ whence
\[
\left\vert T(f)-T(g)\right\vert =\sup\left\{  -\left(  T(f)-T(g\right)
),T(f)-T(g)\right\}  \leq\left\vert T(f-g)\right\vert +\left\vert
T(g-f)\right\vert .
\]
Therefore $\left\Vert T(f)-T(g)\right\Vert \leq\left\Vert T(f-g)\right\Vert
+\left\Vert T(g-f)\right\Vert \leq2\left\Vert T\right\Vert \left\Vert
f-g\right\Vert $ and the proof is done.
\end{proof}

It is a remarkable fact that all sublinear and monotone operators are
automatically continuous. This was first proved by M. G. Krein for positive
linear functionals \cite{Krein} and later generalized in several linear
contexts by various authors (including Klee, Lozanovsky, Namioka and
Schaefer). See \cite{AA2001}.

\begin{theorem}
\label{thmKrein}Every sublinear and monotone operator $T$ $:E\rightarrow F$ is
Lipschitz continuous and the Lipschitz constant of $T$ equals $\left\Vert
T\right\Vert ,$ that is,
\[
\left\Vert T(f)-T(g)\right\Vert \leq\left\Vert T\right\Vert \left\Vert
f-g\right\Vert \text{\quad for all }f,g\in E.
\]

\end{theorem}

\begin{proof}
Notice first that for all $f,g\in E$ we have $f\leq g+\left\vert
f-g\right\vert ,$ whence $T(f)-T(g)\leq T\left(  \left\vert f-g\right\vert
\right)  $ due to the monotonicity and subadditivity of $T.$ Interchanging the
role of $f$ and $g$ we infer that $-\left(  T(f)-T(g)\right)  \leq T\left(
\left\vert f-g\right\vert \right)  .$ Therefore
\begin{equation}
\left\vert T(f)-T(g)\right\vert =\sup\left\{  T(f)-T(g),-\left(
T(f)-T(g)\right)  \right\}  \leq T\left(  \left\vert f-g\right\vert \right)  .
\label{ord-Lip}%
\end{equation}
Now the continuity of $T$ can be established by reductio ad absurdum. Indeed,
if $T$ were not continuous, then would exist a sequence $(x_{n})_{n}$ of
elements of $E$ such that $\left\Vert f_{n}\right\Vert \leq1/\left(
n2^{n}\right)  $ and $\left\Vert T(f_{n})\right\Vert \geq n.$ Taking into
account the inequality $\left\vert T(f_{n})\right\vert \leq T\left(
\left\vert f_{n}\right\vert \right)  $, and replacing each $x_{n}$ by
$\left\vert f_{n}\right\vert $ if necessary, one may restrict ourselves to the
case where all elements $f_{n}$ belong to $E_{+}.$ Then the series $\sum
f_{n}$ is absolutely converging, with sum $f\geq0.$ Since $T$ is monotone,%
\[
T(f)\geq T(f_{n})
\]
which implies $\left\Vert T(f)\right\Vert \geq\left\Vert T(f_{n})\right\Vert
\geq n$ for all positive integers $n,$ a contradiction.

Once the continuity of $T$ was established, we infer from (\ref{ord-Lip})
that
\[
\left\Vert T(f)-T(g)\right\Vert \leq\left\Vert T\left(  \left\vert
f-g\right\vert \right)  \right\Vert \leq\left\Vert T\right\Vert \left\Vert
f-g\right\Vert ,
\]
which ends the proof.
\end{proof}

\begin{remark}
\label{rem-orderLip}The proof of Theorem \ref{thmKrein}, shows that every
sublinear and monotone operator $T$ $:E\rightarrow F$ verifies the condition
$\left\vert T(f)-T(g)\right\vert \leq T\left(  \left\vert f-g\right\vert
\right)  $ for all $f,g$ in $E.$
\end{remark}

\section{The case of almost everywhere convergence}

We start with an extension of Korovkin's theorem in the context of sublinear
and monotone operators acting on the vector lattice $\mathcal{AC}_{b}(X)$, of
all bounded and almost everywhere continuous $f:X\rightarrow\mathbb{R}.$

\begin{theorem}
\label{thm_a.e.Conv}Suppose that $X$ is a locally compact subset of the
Euclidean space $\mathbb{R}^{N}$ endowed with a positive Borel measure $\mu$
and $E$ is a vector sublattice of $\mathcal{F}(X)$ that contains the following
set of test functions: $1,~\pm\operatorname*{pr}_{1},...,~\pm
\operatorname*{pr}_{N}$ and $\sum_{k=1}^{N}\operatorname*{pr}_{k}^{2}$.

$(i)$ If $(T_{n})_{n}$ is a sequence of sublinear and monotone operators from
$E$ into $E$ such that
\[
T_{n}(f)(x)\rightarrow f(x)\text{\quad a.e.}%
\]
for each of the $2N+2$ aforementioned test functions, then this property
extends to all nonnegative functions $f$ in $E\cap\mathcal{AC}_{b}(X)$.

$(ii)$ If, in addition, each operator $T_{n}$ is translatable, then
$T_{n}(f)(x)\rightarrow f(x)$ a.e. for every $f\in E\cap\mathcal{AC}%
_{b}\left(  X\right)  $.

Moreover, in both cases $(i)$ and $(ii)$ the family of test functions can be
reduced to $1,~-\operatorname*{pr}_{1},...,~-\operatorname*{pr}_{N}$ and
$\sum_{k=1}^{N}\operatorname*{pr}_{k}^{2}$ provided that $X$ is included in
the positive cone of $\mathbb{R}^{N}$.
\end{theorem}

\begin{proof}
Let $f\in E\cap\mathcal{AC}_{b}(\Omega)$ and let $\omega$ be a point of
continuity of $f$ which is also a point where
\[
T_{n}(h)(\omega)\rightarrow h(\omega)
\]
for each of the functions $h\in\left\{  1,~\pm\operatorname*{pr}_{1}%
,...,~\pm\operatorname*{pr}_{N}\text{ and }\sum_{k=1}^{N}\operatorname*{pr}%
_{k}^{2}\right\}  .$

Then for $\varepsilon>0$ arbitrarily fixed, there is $\delta>0$ such that
\[
|f(x)-f(\omega)|\leq\varepsilon\text{\quad for every }x\in X\text{ with }\Vert
x-\omega\Vert\leq\delta.
\]

If $\Vert x-\omega\Vert\geq\delta$, then
\[
|f(x)-f(\omega)|\leq\frac{2\Vert f\Vert_{\infty}}{\delta^{2}}\cdot\Vert
x-\omega\Vert^{2},
\]
so that
\begin{equation}
|f(x)-f(\omega)|\leq\varepsilon+\frac{2\Vert f\Vert_{\infty}}{\delta^{2}}%
\cdot\Vert x-\omega\Vert^{2}\text{\quad for all }x\in X. \label{est1}%
\end{equation}

Denoting%
\[
M=\max\left\{  \operatorname*{pr}\nolimits_{1}(\omega),...,\operatorname*{pr}%
\nolimits_{N}(\omega),0\right\}  ,
\]
one can restate (\ref{est1}) as%
\begin{multline*}
\left\vert f(x)-f(\omega)\right\vert \leq\varepsilon+\frac{2\Vert
f\Vert_{\infty}}{\delta^{2}}\left[  \sum_{k=1}^{N}\operatorname*{pr}%
\nolimits_{k}^{2}(x)+2\sum_{k=1}^{N}\operatorname*{pr}\nolimits_{k}%
(x)(M-\operatorname*{pr}\nolimits_{k}(\omega))\right. \\
\left.  +2M\sum_{k=1}^{N}\left(  -\operatorname*{pr}\nolimits_{k}(x)\right)
+\left\Vert \omega\right\Vert ^{2}\right]  .
\end{multline*}
Taking into account Remark \ref{rem-orderLip} and the properties of
sublinearity and monotonicity of the operators $T_{n},$ we infer in the case
where $f\geq0$ that
\begin{align*}
|T_{n}(f)-f(\omega)|  &  \leq\left\vert T_{n}(f)-T_{n}(f(\omega)\cdot
1)+f(\omega)T_{n}(1)-f(\omega)\right\vert \\
&  \leq T_{n}(|f-f(\omega)|)+f(\omega)|T_{n}(1)-1|\\
&  \leq\varepsilon T_{n}(1)+\frac{2\Vert f\Vert_{\infty}}{\delta^{2}}\left[
T_{n}\left(  \sum_{k=1}^{N}\operatorname*{pr}\nolimits_{k}^{2}\right)
+2\sum_{k=1}^{N}(M-\operatorname*{pr}\nolimits_{k}(\omega))T_{n}\left(
\operatorname*{pr}\nolimits_{k}\right)  )\right. \\
&  +\left.  2M\sum_{k=1}^{N}T_{n}\left(  -\operatorname*{pr}\nolimits_{k}%
\right)  +\left\Vert \omega\right\Vert ^{2}T_{n}(1)\right]  +f(\omega
)|T_{n}(1)-1|.
\end{align*}
By our choice of $\omega,$
\[
\underset{n\rightarrow\infty}{\lim\sup}\left\vert T_{n}(f)(\omega
)-f(\omega)\right\vert \leq\varepsilon,
\]
whence we conclude that $T_{n}(f)(\omega)\rightarrow f(\omega)$ since
$\varepsilon>0$ was arbitrarily fixed.

$(ii)$ Suppose in addition that each operator $T_{n}$ is also translatable.
According to the assertion $(i),$
\[
T_{n}(f+\Vert f\Vert_{\infty})\rightarrow f+\Vert f\Vert_{\infty}\text{\quad
a.e.}%
\]
Since the operators $T_{n}$ translatable, $T_{n}(f+\Vert f\Vert_{\infty
})=T_{n}(f)+\left\Vert f\right\Vert _{\infty}T_{n}(1)$ and by our hypotheses
$T_{n}(1)\rightarrow1$ a.e. Therefore $T_{n}(f)\rightarrow f$ a.e.

As concerns the last assertion of Theorem \ref{thm_a.e.Conv}, notice that when
$X$ is included in the positive cone of $\mathbb{R}^{N}$ one can restate the
estimate (\ref{est1}) as%
\begin{multline*}
\left\vert f(x)-f(\omega)\right\vert \leq\varepsilon\\
+\frac{2\Vert f\Vert_{\infty}}{\delta^{2}}\left[  \sum_{k=1}^{N}%
\operatorname*{pr}\nolimits_{k}^{2}(x)+2\sum_{k=1}^{N}(-\operatorname*{pr}%
\nolimits_{k}(x))\operatorname*{pr}\nolimits_{k}(\omega))+\left\Vert
\omega\right\Vert ^{2}\right]  ,
\end{multline*}
which leads to%
\begin{multline*}
|T_{n}(f)-f(\omega)|\leq T_{n}(|f-f(\omega)|)+f(\omega)|T_{n}(1)-1|\\
\leq\varepsilon+\frac{2\Vert f\Vert_{\infty}}{\delta^{2}}\left[  T_{n}\left(
\sum_{k=1}^{N}\operatorname*{pr}\nolimits_{k}^{2}(x)\right)  \right.  +\left.
2\operatorname*{pr}\nolimits_{k}(\omega)\sum_{k=1}^{N}T_{n}\left(
-\operatorname*{pr}\nolimits_{k}(x)\right)  +\left\Vert \omega\right\Vert
^{2}T_{n}(1)\right] \\
+f(\omega)|T_{n}(1)-1|
\end{multline*}
in the case where $f\geq0.$ Then the proof continues verbatim as in the cases
$(i)$ and $(ii).$
\end{proof}

\begin{corollary}
\label{cor1}Let $\mathcal{R}(K)$ be the vector lattice of all Riemann
integrable functions defined on an $N$-dimensional compact interval
$K=\prod\nolimits_{k=1}^{N}[a_{k},b_{k}]$ and let $(T_{n})_{n}$ be a sequence
of sublinear and monotone operators from $\mathcal{R}(K)$ into itself such
that
\[
T_{n}(f)\rightarrow f\text{\quad a.e.}%
\]
for each of the functions $1,~\pm\operatorname*{pr}_{1},...,~\pm
\operatorname*{pr}_{N}~$and $\sum_{k=1}^{N}\operatorname*{pr}_{k}^{2}$. Then
this convergence also occurs for all nonnegative functions $f\in
\mathcal{R}(K).$ It occurs for all Riemann integrable functions defined on $K$
when the operators $T_{n}$ are weakly nonlinear and monotone.
\end{corollary}

\begin{proof}
One applies Theorem \ref{thm_a.e.Conv}, by taking into account Lebesgue's
characterization of Riemann integrable functions: a function $f:K\rightarrow
\mathbb{R}$ belongs to $\mathcal{R}(K)$ if and only if $f$ is bounded and the
set points where it is not continuous has Lebesgue measure zero. See
\cite{CN2014}, p. 323, for the case $N=1$ and \cite{Zor}, pp. 110-113,
Sec.11.1, for an arbitrary $N\geq1$.
\end{proof}

\begin{remark}
\label{rem0}$($The necessity of hypotheses in Theorem \emph{\ref{thm_a.e.Conv}%
}$)$ Though being sufficient for the fulfillment of Theorem
\emph{\ref{thm_a.e.Conv}}, none of the three conditions imposed to the
operators $T_{n}$ $($sublinearity, monotonicity and translatability$)$ is
necessary. See the case of the sequence of operators,
\[
S_{n}:\mathcal{R}([0,1])\rightarrow\mathcal{R}([0,1]),\text{\quad}%
S_{n}(f)=f+f^{2}/n,
\]
which fails all these assumptions despite its converges to the identity of
$\mathcal{R}([0,1])$.
\end{remark}

\begin{remark}
\label{rem1}The proof of Theorem \emph{\ref{thm_a.e.Conv}} still works in the
variant where the space $\mathcal{AC}_{b}(X)$ is replaced by $C_{b}(X)$ and
the convergence a.e. is replaced respectively by pointwise convergence $($or
even by uniform convergence on compact subsets$).$
\end{remark}

\begin{remark}
\label{rem2}Suppose that $(T_{n})_{n}$ is a sequence of weakly linear and
monotone operators from $E$ into $E$ and let $\Omega_{f}$ be a set consisting
of points of continuity of a function $f\in E\cap\mathcal{F}_{b}\left(
X\right)  .$ An inspection of the argument of Theorem \emph{\ref{thm_a.e.Conv}%
} shows that $T_{n}(f)(x)\rightarrow f(x)$ for every $x\in\Omega_{f}$ provided
that%
\[
T_{n}(h)(x)\rightarrow h(x)
\]
for every $x\in\Omega_{f}$ and every test function $h\in\{1,~\pm
\operatorname*{pr}_{1},...,~\pm\operatorname*{pr}_{N},~\sum_{k=1}%
^{N}\operatorname*{pr}_{k}^{2}\}.$

This remark also works in the case of Corollary $1$.
\end{remark}

In connection with Remark \ref{rem2} let us mention the following nonlinear
generalization of a result due to Altomare. See \cite{Alt2021}, Section 2.

\begin{theorem}
\label{thmAltomare}Suppose that $X$ is a subset of the Euclidean space
$\mathbb{R}^{N},$ $\omega$ is point in $X$ and $E$ is a sublattice of
$\mathcal{F}(X)$ that contains the unit $1$ and also the following set of test
functions: $1,~\pm\operatorname*{pr}_{1},...,~\pm\operatorname*{pr}_{N}$ and
$\sum_{k=1}^{N}\operatorname*{pr}_{k}^{2}$.

If $(T_{n})_{n}$ is a sequence of sublinear and monotone functionals defined
on $E$ such that
\begin{equation}
T_{n}(f)(x)\rightarrow f(\omega)\text{\quad for }x\in X
\end{equation}
whenever $f$ is one of the test functions $1,~\pm\operatorname*{pr}%
_{1},...,~\pm\operatorname*{pr}_{N}$ and $\sum_{k=1}^{N}\operatorname*{pr}%
_{k}^{2}$, then
\begin{equation}
\lim_{n\rightarrow\infty}T_{n}(f)=f(\omega)\text{\quad for }x\in X
\end{equation}
for all nonnegative functions $f\in\mathcal{F}_{b}(X)$ which are continuous at
$\omega.$ The conclusion occurs for all functions $f\in\mathcal{F}_{b}(X)$
continuous at $\omega$ when the functionals $T_{n}$ are weakly nonlinear and monotone.
\end{theorem}

The proof is similar to that of Theorem \ref{thm_a.e.Conv} and the details are
left to the reader as an exercise.

\section{The case of convergence in measure and of convergence in $L^{p}%
$-norm}

The convergence almost everywhere is related to other modes of convergence.
Indeed, the convergence almost everywhere implies \emph{local convergence in
measure}, that is, $f_{n}\rightarrow f$ a.e. implies
\[
\lim_{n\rightarrow\infty}\mu\left(  \left\{  x\in A:\left\vert f_{n}%
(x)\rightarrow f(x)\right\vert \geq\varepsilon\right\}  \right)  =0
\]
for every $\varepsilon>0$ and every Borel set $A$ with $\mu(A)<\infty.$ The
converse fails but if $\mu$ is $\sigma$-finite, then $(f_{n})_{n}$ converges
to $f$ locally in measure if and only if every subsequence has in turn a
subsequence that converges to $f$ almost everywhere.

In the next section we will discuss the connection between the global
convergence in measure and the convergence in $p$-mean. Recall that
$f_{n}\rightarrow f$ \emph{globally in measure} if%
\[
\lim_{n\rightarrow\infty}\mu\left(  \left\{  x\in X:\left\vert f_{n}%
(x)\rightarrow f(x)\right\vert \geq\varepsilon\right\}  \right)  =0
\]
for every $\varepsilon>0.$ When $\mu(X)<\infty,$ the two types of convergence
in measure coincide and we will refer to each of them as \emph{convergence in
measure. }In the same context, every sequence of measurable functions $f_{n}$
that converges almost everywhere to a function $f$, also converges to $f$ in
measure. The converse assertion is false: there exists a sequence of
measurable functions on $[0,1$] that converges to zero in Lebesgue measure but
does not converge at any point at all. See \cite{Bog}, Theorem 2.2.3, p. 111
and Example 2.2.4, p.112.

The analogue of Theorem \ref{thm_a.e.Conv} in the case of convergence in
measure is as follows.

\begin{theorem}
\label{thm-measure}Suppose that $X$ is a compact subset of the Euclidean space
$\mathbb{R}^{N}$ endowed with a positive Borel measure $\mu$ and let
$(T_{n})_{n}$ be a sequence of sublinear and monotone operators from $C(X)$
into itself such that
\[
\lim_{n\rightarrow\infty}\mu\left(  \left\{  x\in X:\left\vert T_{n}%
(f)\rightarrow f\right\vert \geq\varepsilon\right\}  \right)  =0
\]
for each of the functions $1,~\pm\operatorname*{pr}_{1},...,~\pm
\operatorname*{pr}_{N}~$and $\sum_{k=1}^{N}\operatorname*{pr}_{k}^{2}$ and
each $\varepsilon>0.$

Then this convergence occurs for all nonnegative functions $f\in C(X).$ It
occurs for all functions in $C(X)$ provided that the operators $T_{n}$ are
weakly nonlinear and monotone.
\end{theorem}

\begin{proof}
Let $f\in C(X)$ be a nonnegative function and let $\varepsilon>0$ arbitrarily
fixed. Due to the uniform continuity of the function $f,$ for every $\alpha
\in(0,\varepsilon/2)$ there is $\delta>0$ such that
\[
|f(t)-f(x)|\leq\alpha+\delta\left\Vert t-x\right\Vert ^{2}\quad\text{for all
}t,x\in X;
\]
reason by reductio ad absurdum. Proceeding as in the proof of Theorem
\ref{thm_a.e.Conv} we infer that for all $x\in X$ and $n\in\mathbb{N}$ we
have
\begin{multline*}
|T_{n}(f)-f(x)|\leq\left\vert T_{n}(f)-T_{n}(f(x)\cdot1)+f(x)T_{n}%
(1)-f(x)\right\vert \\
\leq T_{n}(|f-f(x)|)+f(x)|T_{n}(1)-1|\\
\leq\alpha+\delta\left[  T_{n}\left(  \sum_{k=1}^{N}\operatorname*{pr}%
\nolimits_{k}^{2}\right)  +2\sum_{k=1}^{N}(M-\operatorname*{pr}\nolimits_{k}%
(x))T_{n}\left(  \operatorname*{pr}\nolimits_{k}\right)  )\right. \\
+\left.  2M\sum_{k=1}^{N}T_{n}\left(  -\operatorname*{pr}\nolimits_{k}\right)
+\left\Vert x\right\Vert ^{2}T_{n}(1)\right]  +f(x)|T_{n}(1)-1|,
\end{multline*}
where $M=\max\left\{  \operatorname*{pr}\nolimits_{1}(\omega
),...,\operatorname*{pr}\nolimits_{N}(\omega),0:x\in X\right\}  .$

Choose a rank $N$ such that $\left\Vert f\right\Vert _{\infty}\left\vert
T_{n}(1)-1\right\vert \leq\alpha$ for every $n\geq N.$ Then for $n\geq N$ the
set
\[
\{x\in X:|T_{n}(f)(x)-f(x)|\geq\varepsilon\}
\]
is included in the set of points $x\in X$ where%
\begin{multline*}
T_{n}\left(  \sum_{k=1}^{N}\operatorname*{pr}\nolimits_{k}^{2}\right)
(x)+2\sum_{k=1}^{N}(M-\operatorname*{pr}\nolimits_{k}(x))T_{n}\left(
\operatorname*{pr}\nolimits_{k}\right)  (x))\\
+2M\sum_{k=1}^{N}T_{n}\left(  -\operatorname*{pr}\nolimits_{k}\right)
(x)+\left\Vert x\right\Vert ^{2}T_{n}(1)(x)\geq\frac{\varepsilon-2\alpha
}{\delta}.
\end{multline*}
This implies that
\begin{align*}
\mu(\{x  &  \in X:|T_{n}(f)-f|\geq\varepsilon\})\\
&  \leq\mu\left(  \left\{  x\in X:T_{n}(\sum\nolimits_{k=1}^{N}%
\operatorname*{pr}\nolimits_{k}^{2})+\cdots\geq\left(  \varepsilon
-2\alpha\right)  /\delta\right\}  \right)
\end{align*}
for every $n\geq N.$ Taking into account our hypothesis and the fact that the
sum of sequences convergent in measure is also a sequence convergent in
measure (see, Bogachev \cite{Bog}, Corollary 2.2.6, p. 113) we conclude the
proof of the first part of Theorem \ref{thm-measure}.

For the second part, if $f\in C(X)$ is an arbitrary real valued function we
will apply the preceding reasoning to $f+\Vert f\Vert_{\infty}\geq0$ to infer
that $T_{n}(f+\Vert f\Vert_{\infty})\rightarrow f+\Vert f\Vert_{\infty}$ in
measure. We have%
\[
T_{n}(f+\Vert f\Vert_{\infty})=T_{n}(f)+\Vert f\Vert_{\infty}T_{n}(1)
\]
because the operators $T_{n}$ are assumed to be weakly nonlinear and
$T_{n}(1)\rightarrow1$ by our hypotheses. The proof ends by using again the
algebraic operations with sequences convergent in measure.
\end{proof}

\begin{remark}
When $X$ is locally compact, then the statement of Theorem
\emph{\ref{thm-measure}} still works if global convergence in measure is
replaced by the following condition of local convergence,%
\[
\lim_{n\rightarrow\infty}\mu\left(  \left\{  x\in K:\left\vert T_{n}%
(f)\rightarrow f\right\vert \geq\varepsilon\right\}  \right)  =0
\]
for every $\varepsilon>0$ and every compact subset $K$ of $X.$
\end{remark}

\begin{remark}
As in the case of Theorem \emph{\ref{thm_a.e.Conv}}, none of the three
conditions imposed to the operators $T_{n}$ $($sublinearity, monotonicity and
translatability$)$ is necessary for the fulfillment of Theorem
\emph{\ref{thm-measure}.} See the example offered by Remark \emph{\ref{rem0}}.
\end{remark}

We continue this section by considering the case of convergence in $L^{p}%
$-norm. The existing literature includes many papers containing various
generalization of Korovkin's theorem in the context of operators acting on
$L^{p}$-spaces. See Altomare \cite{Alt2010}, Altomare and Campiti
\cite{AC1994}, Berens and DeVore\cite{BD1}, \cite{BD2}, Donner \cite{Don1981},
\cite{Don2006}, Swetits and Wood \cite{SW}, Wulbert \cite{Wul} to cite just a
few. However, all of them refer to the case of linear and positive operators.
The next theorem provides a nonlinear generalization bases on sequences of
weakly nonlinear and monotone operators.

\begin{theorem}
\label{thmLp}Let $\mu$ be a positive Borel measure on $\mathbb{R}^{N}$ with
compact support and let $(T_{n})_{n}$ be a sequence of sublinear and monotone
operators from the Banach lattice $L^{p}(\mu)$ into itself, where $p\in
\lbrack1,\infty).$ If $M=\sup_{n}\left\Vert T_{n}\right\Vert <\infty$ and
\[
T_{n}(f)\rightarrow f\text{\quad in }p\text{-mean}%
\]
for each of the test functions $1,~\pm\operatorname*{pr}_{1},...,~\pm
\operatorname*{pr}_{N}~$ and $\sum_{k=1}^{N}\operatorname*{pr}_{k}^{2},$ then
this convergence occurs for all nonnegative functions $f\in L^{p}(\mu).$ It
occurs for all functions in $L^{p}(\mu)$ provided that the operators $T_{n}$
are weakly nonlinear and monotone.
\end{theorem}

\begin{proof}
Let $f\in L^{P}(\mu),$ $f\geq0.$ Since $C_{c}(\mathbb{R}^{N})$ is dense into
$L^{p}(\mu)$ in the $L^{p}$-norm (see \cite{Bog}, Corollary 4.2.2, p. 252) and
the lattice operations in a Banach lattice are continuous (\cite{MN},
Proposition 1.1.6, p. 6), it follows that for every $\varepsilon>0$ there
exists a nonnegative continuous function $g$ with compact support
$\operatorname*{supp}g$ such that
\[
\left\Vert f-g\right\Vert _{L^{p}}<\varepsilon.
\]

Then
\begin{align*}
\Vert T_{n}(f)-f\Vert_{L^{p}}  &  \leq\Vert T_{n}(f)-T_{n}(g)\Vert_{L^{p}%
}+\Vert T_{n}(g)-gT_{n}(1)\Vert_{L^{p}}\\
&  +\left\Vert gT_{n}(1)-g\right\Vert _{L^{p}}+\Vert g-f\Vert_{L^{p}}%
\end{align*}%
\[
\leq\Vert T_{n}\Vert\left\Vert f-g\right\Vert _{L^{p}}+\Vert T_{n}%
(g)-gT_{n}(1)\Vert_{L^{p}}+\left\Vert g\right\Vert _{\infty}\left\Vert
T_{n}(1)-1\right\Vert _{L^{p}}+\Vert f-g\Vert_{L^{p}}%
\]%
\[
\leq\varepsilon(1+M)+\Vert T_{n}(|g-g(x)|)(x)\Vert_{L^{p}}+\left\Vert
g\right\Vert _{\infty}\left\Vert T_{n}(1)-1\right\Vert _{L^{p}}.
\]

Using the uniform continuity of $g$ one can infer (by reductio ad absurdum)
that there exists a number $\delta\geq\left\Vert g\right\Vert _{\infty}+1$
such that
\[
|g(y)-g(x)|\leq\varepsilon+\delta\cdot\left\Vert y-x\right\Vert ^{2}%
\]
for all $x\ $and $y$ in an open and bounded neighborhood of the support of
$g,$ say $\left\{  z:d(z,\operatorname*{supp}g<1\right\}  $ and therefore for
all $x\ $and $y$ in $\mathbb{R}^{N}.$ Here $\Vert\cdot\Vert$ denotes the
Euclidean norm in $\mathbb{R}^{N}$.

According to our hypotheses there exists a rank $n_{0}$ such that
\[
\Vert T_{n}(h)-h\Vert_{L^{p}}<\varepsilon/\delta
\]
for all $n\geq n_{0}$ and all $h\in\{1,\pm\operatorname*{pr}_{1}%
,...,\pm\operatorname*{pr}_{N},$~$\sum_{k=1}^{N}\operatorname*{pr}_{k}^{2}\}$.

Put
\[
\alpha=\max_{x\in K}\left\{  \operatorname*{pr}\nolimits_{1}%
(x),...,\operatorname*{pr}\nolimits_{N}(x),0\right\}  .
\]
where $K$ denotes the support of $\mu$. We have
\begin{align*}
\Vert T_{n}(|g-g(x)|)(x)\Vert_{L^{p}}  &  \leq\Vert T_{n}(\varepsilon
)(x)+\delta T_{n}(\Vert y-x\Vert^{2})(x)\Vert_{L^{p}}\\
&  \leq\varepsilon\Vert T_{n}(1)\Vert_{L^{p}}+\delta\Vert T_{n}(\Vert
y-x\Vert^{2})(x)\Vert_{L^{p}}\\
&  \leq\varepsilon(\Vert T_{n}(1)-1\Vert_{L^{p}}+\mu\left(  K\right)
)+\delta\Vert T_{n}(\Vert y-x\Vert^{2})(x)\Vert_{L^{p}}\\
&  \leq\varepsilon\left(  \varepsilon+\mu(K)\right)  +\delta\Vert T_{n}(\Vert
y-x\Vert^{2})(x)\Vert_{L^{p}}.
\end{align*}
Since
\[
\Vert y-x\Vert^{2}=\sum_{k=1}^{N}\operatorname*{pr}\nolimits_{k}^{2}%
(y)+2\sum_{k=1}^{N}(\alpha-\operatorname*{pr}\nolimits_{k}%
(x))\operatorname*{pr}\nolimits_{k}(y)+\sum_{k=1}^{N}\operatorname*{pr}%
\nolimits_{k}^{2}(x)+2\alpha\sum_{k=1}^{N}\left(  -\operatorname*{pr}%
\nolimits_{k}(y)\right)  ,
\]
it follows that
\begin{multline*}
T_{n}(\Vert y-x\Vert^{2})(x)\leq T_{n}(\sum_{k=1}^{N}\operatorname*{pr}%
\nolimits_{k}^{2})(x)+2\sum_{k=1}^{N}(\alpha-\operatorname*{pr}\nolimits_{k}%
(x))T_{n}(\operatorname*{pr}\nolimits_{k})(x)\\
+2\alpha\sum_{k=1}^{N}T_{n}(-\operatorname*{pr}\nolimits_{k})(x)+\sum
_{k=1}^{N}\operatorname*{pr}\nolimits_{k}^{2}(x)T_{n}(1)(x)\\
=T_{n}(\sum_{k=1}^{N}\operatorname*{pr}\nolimits_{k}^{2})(x)-\sum_{k=1}%
^{N}\operatorname*{pr}\nolimits_{k}^{2}+2\sum_{k=1}^{N}(\alpha
-\operatorname*{pr}\nolimits_{k}(x))[T_{n}(\operatorname*{pr}\nolimits_{k}%
)(x)+T_{n}(-\operatorname*{pr}\nolimits_{k})(x)]\\
+2\sum_{k=1}^{N}\operatorname*{pr}\nolimits_{k}(x)[T_{n}(-\operatorname*{pr}%
\nolimits_{k})(x)+\operatorname*{pr}\nolimits_{k}(x)]-\sum_{k=1}%
^{N}\operatorname*{pr}\nolimits_{k}^{2}(x)+\sum_{k=1}^{N}\operatorname*{pr}%
\nolimits_{k}^{2}(x)T_{n}(1)(x),
\end{multline*}
whence
\begin{multline*}
\delta\left\Vert T_{n}(\Vert y-x\Vert^{2}(x)\right\Vert _{L^{p}}\leq
\delta\Vert T_{n}(\sum_{k=1}^{N}\operatorname*{pr}\nolimits_{k}^{2}%
)(x)-\sum_{k=1}^{N}\operatorname*{pr}\nolimits_{k}^{2}(x)\Vert_{L^{p}}\\
+4\alpha\delta\sum_{k=1}^{N}\Vert T_{n}(-\operatorname*{pr}\nolimits_{k}%
)(x)+\operatorname*{pr}\nolimits_{k}(x)\Vert_{L^{p}}+4\alpha\delta\sum
_{k=1}^{N}||T_{n}(\operatorname*{pr}\nolimits_{k})(x)-\operatorname*{pr}%
\nolimits_{k}(x)\Vert_{L^{p}}\\
+2\alpha\delta\sum_{k=1}^{N}\Vert T_{n}(-\operatorname*{pr}\nolimits_{k}%
)(x)+\operatorname*{pr}\nolimits_{k}(x)\Vert_{p}+\alpha^{2}\delta\sum
_{k=1}^{N}\Vert T_{n}(1)(x)-1\Vert_{L^{p}}\\
\leq\varepsilon(1+10\alpha+\alpha^{2})
\end{multline*}
for $n$ sufficiently large. This ends the proof in the case where $f\geq0.$

In the general case, notice first that $T_{n}(f)\rightarrow f$ in $p$-mean for
every $f\in L^{p}(\mu)$ with $\left\Vert f\right\Vert _{\infty}<\infty.$
Indeed, $T_{n}\left(  f+\Vert f\Vert_{\infty}\right)  \rightarrow f+\Vert
f\Vert_{\infty}$ by the discussion above and $T_{n}\left(  f+\Vert
f\Vert_{\infty}\right)  =T_{n}(f)+\Vert f\Vert_{\infty}T_{n}(1)$ due to the
fact that the operators $T_{n}$ are weakly nonlinear.

If $f$ is an arbitrary function in $L^{p}(\mu),$ then it can be approximated
by step functions. Every step function $h$ is bounded in the sup-norm, so that
$\left\Vert T_{n}(h)-h\right\Vert _{L^{p}}\rightarrow0.$ Finally, the
inequality%
\begin{multline*}
\left\Vert T_{n}(f)-f\right\Vert \leq\left\Vert T_{n}(f)-T_{n}(h)\right\Vert
_{L^{p}}+\left\Vert T_{n}(h)-h\right\Vert _{L^{p}}+\left\Vert h-f\right\Vert
_{L^{p}}\\
\leq M\left\Vert f-h\right\Vert _{L^{p}}+\left\Vert T_{n}(h)-h\right\Vert
_{L^{p}}+\left\Vert h-f\right\Vert _{L^{p}},
\end{multline*}
allows us to conclude that $\left\Vert T_{n}(f)-f\right\Vert _{L^{p}%
}\rightarrow0.$
\end{proof}

It is conceivable that Theorem \ref{thmLp} still works in the case of an
arbitrary finite Borel measure on $\mathbb{R}^{N},$ but at the moment we lack
a valid argument.

In the setting of linear and continuous operators acting on $L^{1}[0,1],$
Wulbert \cite{Wul} has proved a Korovkin type theorem that avoids the
hypothesis of monotonicity. It is an open question whether his result admits
an analogue within the framework of continuous sublinear operators
$T_{n}:L^{1}[0,1]\rightarrow L^{1}[0,1]$.

\section{An application to Lebesgue differentiation theorem}

The Lebesgue differentiation theorem is an important result in real analysis
that can be stated as follows.

\begin{theorem}
\label{thmLebesgue}Let $f\in L^{1}(\Omega)$ be the Lebesgue space associated
to an open subset $\Omega$ of $\mathbb{R}^{N}$. Consider any collection
$\mathcal{R}$ of closed $N$-dimensional intervals with sides parallel to the
axes, with non-empty interior, centered at the origin $0$ and containing
sequences $(R_{n})_{n}$ contracting to $0$ as $\operatorname*{diam}%
R_{n}\rightarrow0$. Assume furthermore that the intervals in $\mathcal{R}$ are
comparable, i.e., for any two $R_{i}$, $R_{j}\in\mathcal{R}$ either
$R_{i}\subset R_{j}$ or $R_{j}\subset R_{i}$. $($Example: the collection of
all $N$-dimensional closed cubic intervals centered at $0$.$)$

Then, for almost every $x$ and for every $(R_{n})_{n}\subset\mathcal{R}$ with
$\operatorname*{diam}R_{n}\rightarrow0$ one has%
\[
\lim_{n\rightarrow\infty}\frac{1}{\operatorname*{vol}(R_{n})}\int_{R_{n}%
}f(x-y)dy=f(x)\text{\quad a.e.}%
\]

\end{theorem}

A simple proof of this theorem can be found in a paper by de Guzman and Rubio
\cite{GR}. In what follows we present an alternative argument based on our
Theorem \ref{thm_a.e.Conv}.

Indeed, we deal here with the sequence of linear and positive operators
$T_{n}:L^{1}(\Omega)\rightarrow L^{1}(\Omega)$ defined by the formula%
\[
T_{n}f(x)=\frac{1}{\operatorname*{vol}(R_{n})}\int_{R_{n}}f(x-y)dy.
\]
Clearly, these operators are weakly nonlinear. Theorem \ref{thm_a.e.Conv}
applies to those functions $f$ which are bounded in the sup-norm but this can
be easily arranged by approximating an arbitrary $f\in L^{1}(\Omega)$ by step
functions $f_{\varepsilon}$ and noticing that
\begin{multline*}
\int_{R_{n}}\left\vert \frac{1}{\operatorname*{vol}(R_{n})}\int_{R_{n}%
}f(x-y)dy-\frac{1}{\operatorname*{vol}(R_{n})}\int_{R_{n}}f_{\varepsilon
}(x-y)dy\right\vert dx\\
\leq\int_{R_{n}}\left\vert f-f_{\varepsilon}\right\vert dx<\varepsilon
\end{multline*}
for all $n.$ Taking into account that every sequence converging in mean has a
subsequence converging almost everywhere, we may reduce the proof of Theorem
\ref{thmLebesgue} to the case where $f$ vanishes outside a compact set and
$\sup_{x}\left\vert f(x)\right\vert <\infty.$ Under these circumstances the
verification of the property of a.e. convergence for the test functions
$1,~\operatorname*{pr}_{1},...,~\operatorname*{pr}_{N}~$and $\sum_{k=1}%
^{N}\operatorname*{pr}_{k}^{2}$ is a simple exercise. For example, in the one
dimensional case one has to observe that
\begin{align*}
\frac{1}{R_{n}-r_{n}}\int_{r_{n}}^{R_{n}}\mathrm{d}y  &  =1\\
\frac{1}{R_{n}-r_{n}}\int_{r_{n}}^{R_{n}}(x-y)\mathrm{d}y  &  =\frac
{(x-r_{n})^{2}-\left(  x-R_{n}\right)  ^{2}}{2\left(  R_{n}-r_{n}\right)
}=x-\frac{R_{n}+r_{n}}{2}\rightarrow x\\
\frac{1}{R_{n}-r_{n}}\int_{r_{n}}^{R_{n}}(x-y)^{2}\mathrm{d}y  &
=\frac{\left(  x-r_{n}\right)  ^{3}-(x-R_{n})^{3}}{3\left(  R_{n}%
-r_{n}\right)  }\\
&  =x^{2}-x(R_{n}+r_{n})+\frac{R_{n}^{2}+R_{n}r_{n}+r_{n}^{2}}{3}\rightarrow
x^{2}%
\end{align*}
for every $x\in\Omega$ and every sequence of intervals $[r_{n},R_{n}]$ with
$r_{n}<0<R_{n}$ that contracts to the origin.

\begin{remark}
The last computations allow us to construct counterexamples showing that the
convergences asserted by Theorem \emph{\ref{thm_a.e.Conv}}, Theorem
\emph{\ref{thm-measure}} and Theorem \emph{\ref{thmLp}} fail for functions of
variable sign in the absence of the condition of translatability. See the case
of sublinear and monotone operators $T_{n}:L^{1}(0,1)\rightarrow L^{1}(0,1)$
defined by
\[
T_{n}f(x)=\frac{1}{R_{n}-r_{n}}\max\left\{  \int_{r_{n}}^{R_{n}}%
f(x-y)\mathrm{d}y,0\right\}  .
\]

\end{remark}

\section{The convergence of some sequences of Bernstein type operators}

In this section we present some concrete examples illustrating the above
results in the context of Choquet's nonlinear integral. The necessary
background on this integral is covered by the Appendix at the end of our paper
\cite{Gal-Nic-RACSAM}.

We start by considering the \emph{Bernstein-Kantorovich-Choquet polynomial
operators} for functions of one real variable,
\[
K_{n,\mu}^{(1)}:\mathcal{R}([0,1])\rightarrow\mathcal{R}([0,1]),
\]
defined by the formula
\begin{equation}
K_{n,\mu}^{(1)}(f)(x)=\sum_{k=0}^{n}p_{n,k}(x)\cdot\frac{(C)\int
_{k/(n+1)}^{(k+1)/(n+1)}f(t)\mathrm{d}\mu}{\mu([k/(n+1),(k+1)/(n+1)])},
\label{BKC}%
\end{equation}
where the symbol $(C)$ in front of the integral means that we deal with a
Choquet integral (with respect to the capacity $\mu$). As usually in
approximation theory,
\[
p_{n,k}(t)={\binom{n}{k}}t^{k}(1-t)^{n-k},\text{\quad for }t\in\lbrack
0,1]\text{ and }n\in\mathbb{N}.
\]
Due to the properties of Choquet's integral, when $\mu$ is a submodular
capacity, it follows that each operator $K_{n,\mu}^{(2)}$ is weakly nonlinear,
monotone and unital from $\mathcal{R}([0,1]$ into itself. This happens in
particular when $\mu$ is the Lebesgue measure $\mathcal{L}$, in which case the
Choquet integral reduces to Lebesgue integral and the
\emph{Bernstein-Kantorovich-Choquet polynomial operators }coincide with the
\emph{Bernstein-Kantorovich polynomial operators}%
\[
K_{n}^{(1)}(f)(x)=(n+1)\sum_{k=0}^{n}p_{n,k}(x)\cdot\int_{k/(n+1)}%
^{(k+1)/(n+1)}f(t)\mathrm{d}t
\]
which act also on $\mathcal{R}([0,1]).$

It is known that $K_{n}^{(1)}(x^{k})\rightarrow x^{k}$\ uniformly on $[0,1]$
for $k\in\{0,1,2\},$ so from Remark \ref{rem2} we infer the following result
previously noticed by Lorentz \cite{Lor}, Theorem 2.1.1, p. 30:

\begin{theorem}
\label{thmLorentz}$K_{n}^{(1)}(f)(x)\rightarrow f(x)$ at each point of
continuity of $f\in\mathcal{R}([0,1])$ $($and thus $K_{n}^{(1)}%
(f)(x)\rightarrow f(x)$ a.e. on $[0,1]).$
\end{theorem}

Theorem \ref{thmLorentz} extends verbatim to the case of tensor product
multivariate Bernstein-Kantorovich polynomial operators, by using the general
Theorem in Haussmann-Pottinger \cite{HP1977}, page 213.

A similar result works for the Bernstein-Kantorovich-Choquet polynomial
operators $K_{n,\sqrt{\mathcal{L}}}^{(1)}$ associated to the submodular
capacity $\mu=\sqrt{\mathcal{L}}:A\rightarrow\sqrt{\mathcal{L}(A)}$ for
$A\in\mathcal{B}([0,1]).$ Indeed, as we noticed in \cite{Gal-Nic-Med}, Section
3,
\[
K_{n,\sqrt{\mathcal{L}}}^{(1)}(-x)\rightarrow-x\text{ and }K_{n,\sqrt
{\mathcal{L}}}^{(1)}(x^{k})\rightarrow x^{k}\text{\ uniformly on }[0,1]
\]
for $k\in\{0,1,2\}$.

\begin{remark}
According to Remark \emph{\ref{rem2}}, it follows that
\[
K_{n,\sqrt{\mathcal{L}}}^{(1)}(f)(x)\rightarrow f(x)
\]
at each point $x$ of continuity of $f\in\mathcal{R}([0,1])$.
\end{remark}

The Bernstein-Kantorovich-Choquet polynomial operators for functions of two
real variables are defined by the formula
\begin{multline*}
K_{n,\mu}^{(2)}(f)(x_{1},x_{2})=\sum_{k_{1}=0}^{n}\sum_{k_{2}=0}^{n}%
p_{n,k_{1}}(x_{1})p_{n,k_{2}}(x_{2})\\
\cdot\frac{(C)\int_{k_{1}/(n+1)}^{(k_{1}+1)/(n+1)}\left(  (C)\int
_{k_{2}/(n+1)}^{(k_{2}+1)/(n+1)}f(t_{1},t_{2})\mathrm{d}\mu(t_{2})\right)
\mathrm{d}\mu(t_{1})}{\mu([k_{1}/(n+1),(k_{1}+1)/(n+1)])\mu([k_{2}%
/(n+1),(k_{2}+1)/(n+1)])},
\end{multline*}
and they are weakly linear, monotone and unital operators from $\mathcal{R}%
([0,1]^{2})$ provided that the capacity $\mu$ is submodular.

\begin{remark}
Since
\begin{equation}
K_{n,\sqrt{\mathcal{L}}}^{(2)}(f)(x_{1},x_{2})\rightarrow f(x_{1}%
,x_{2})\text{\quad uniformly on }[0,1]^{2}, \label{convBKC}%
\end{equation}
for each of the test functions \thinspace$1,~\pm\operatorname*{pr}%
\nolimits_{1},~\pm\operatorname*{pr}\nolimits_{2},~\operatorname*{pr}%
\nolimits_{1}^{2}+\operatorname*{pr}\nolimits_{2}^{2}$, it follows from Remark
\emph{\ref{rem2} }that
\[
K_{n\sqrt{\mathcal{L}},\mu}^{(2)}(f)(x_{1},x_{2})\rightarrow f(x_{1},x_{2})
\]
at each point of continuity of the function $f\in\mathcal{R}([0,1]^{2}).$ Here
$\mathcal{L}$ is the planar Lebesgue measure.
\end{remark}

Consider now the case of the Sz\'{a}sz-Mirakjan-Kantorovich operators, acting
on the space $\mathcal{R}_{loc,b}([0,\infty))$ of all functions bounded on
$[0,+\infty)$ and Riemann integrable on each compact subinterval, by the
formula
\[
S_{n}(f)(x)=(n+1)e^{-nx}\sum_{k=0}^{\infty}\frac{(nx)^{k}}{k!}\int
_{k/n}^{(k+1)/n}f(t)\mathrm{d}x.
\]
It is known that
\[
S_{n}(x^{k})\rightarrow x^{k}\text{\ pointwise for }x\in\text{ }[0,\infty),
\]
whenever $k\in\{0,1,2\}$. See Walczak \cite{Wal}. Taking into account Remark
\ref{rem2} we obtain the following result that seems to be now:

\begin{theorem}
If $f\in\mathcal{R}_{loc,b}([0,\infty)),$ then $S_{n}(f)(x)\rightarrow f(x)$
at each point of continuity of $f$ $($and thus it converges everywhere on
$[0,+\infty)).$
\end{theorem}

\begin{remark}
A Choquet companion to the Sz\'{a}sz-Mirakjan-Kantorovich operators is
provided by the following sequence of operators:
\[
S_{n,\sqrt{\mathcal{L}}}(f)(x)=e^{-nx}\sum_{k=0}^{\infty}\frac{(C)\int
_{k/n}^{(k+1)/n}f(t)\mathrm{d}\sqrt{\mathcal{L}}}{\mu([k/n,(k+1)/n])}%
\cdot\frac{(nx)^{k}}{k!},\text{\quad}f\in\mathcal{R}_{loc,b}([0,\infty)),
\]
where the integration is performed with respect to the capacity $\sqrt
{\mathcal{L}}.$ According to \emph{\cite{Gal-Nic-Med}},
\[
S_{n,\sqrt{\mathcal{L}}}(h)(x)\rightarrow h(x),
\]
pointwise for all points $x\in\lbrack0,+\infty)$ and test functions
$h\in\{1,x,-x,x^{2}\}$. Therefore, $S_{n,\sqrt{\mathcal{L}}}(f)(x)\rightarrow
h(x)$, pointwise at each point of continuity of $f\in\mathcal{R}%
_{loc,b}([0,\infty)).$
\end{remark}

\section{Further results and comments}

\noindent As was noticed by Korovkin \cite{Ko1953}, \cite{Ko1960}, his theorem
mentioned in the Introduction also works when the unit interval is replaced by
the unit circle
\[
S^{1}=\left\{  (\cos\theta,\sin\theta):\theta\in\lbrack0,2\pi)\right\}
\subset\mathbb{R}^{2}.
\]
Our results show that more is true:

\begin{theorem}
\label{thmtrig}Let $E$ be a linear subspace of continuous real-valued
functions defined on the unit circle that contains the test functions
$1,\cos,-\cos,\sin$ and $-\sin.$ If $(T_{n})_{n}$ is a sequence of weakly
nonlinear and monotone operators which carry $E$ into the space $C(S^{1})$
such that $T_{n}(f)\rightarrow f$ converges almost everywhere $($respectively,
in measure or in $p$-mean$)$, for each test function, then the same mode of
convergence holds for every $f\in E.$
\end{theorem}

The proof follows from Theorems 1-3, by remarking that the cosine function can
be seen as the restriction of $\operatorname*{pr}_{1}$ to $S^{1},$ and the
sine function as the restriction of $\operatorname*{pr}_{2}$ to $S^{1}.$

We leave to the reader the extension of Theorem \ref{thmtrig} in some other
cases of interest such as the torus $S^{1}\times S^{1}$ and the $2$%
-dimensional sphere $S^{2}$.

\begin{remark}
$($Mixing metric spaces and spherical domains$)$\label{exPopa} Suppose that
$X$ is a compact subset of $\mathbb{R}^{N}$. Then the product space $X\times
S^{1}$ is a compact subset of $\mathbb{R}^{N+1}$ and the space $C(X\times
S^{1})$ can be identified with the Banach space $C_{2\pi}(X\times\mathbb{R}),$
of all continuous functions $f:X\times\mathbb{R}\rightarrow\mathbb{R}$, $2\pi
$-periodic in the second variable, endowed with the sup norm.

The reader can easily check that our Theorems $1$-$3$ extend to the case of
sequences of weakly nonlinear operators and monotone\emph{ }operators\emph{
}$T:C(X\times S^{1})\rightarrow C(X\times S^{1})$ and the following set of
test functions $f(x)=u(x)v(\varphi),$ where%
\[
u\in\left\{  1,~\pm\operatorname*{pr}\nolimits_{1},...,~\pm\operatorname*{pr}%
\nolimits_{N}~\text{and}\sum_{k=1}^{N}\operatorname*{pr}\nolimits_{k}%
^{2}\right\}  \text{ and }v\in\left\{  1,~\pm\cos\varphi~,~\pm\sin
\varphi\right\}  .
\]
The case of uniform convergence has been noticed in
\emph{\cite{Gal-Nic-RACSAM}} $($extending the case of positive linear
operators settled in \emph{\cite{Popa}}$).$
\end{remark}

All our theorems remain valid in the context of Ces\`{a}ro convergence. If
$(T_{n})_{n}$ is a sequence of weakly nonlinear and monotone operators (from a
Banach lattice of functions $E$ into itself), then so is the sequence
$(\frac{1}{n}\sum\nolimits_{k=1}^{n}T_{k})_{n}$. As a consequence, adding the
conditions imposed by Theorems 1-3, we infer that
\[
\frac{1}{n}\sum\nolimits_{k=1}^{n}T_{k}(f)\rightarrow f
\]
almost everywhere $($respectively, in measure or in $p$-mean$)$ for every
$f\in E$ whenever it happens for the functions $1,~\pm\operatorname*{pr}%
_{1},...,~\pm\operatorname*{pr}_{N},$ $\sum_{k=1}^{N}\operatorname*{pr}%
_{k}^{2}.$

Last but not the least, one can extend our results (following the model of
Theorem 2 in \cite{Gal-Nic-RACSAM}) using instead of the classical families of
test functions on $\mathbb{R}^{N}$ the separating functions, which allow us to
replace the critical inequalities of the form (\ref{est1}) by inequalities
that work in the general context of metric spaces. For details concerning
these functions see \cite{N2009}.

\end{document}